\documentclass[a4paper,10.5pt,twoside,dvips]{amsart}
\usepackage{amssymb}
\usepackage{amsfonts}
\usepackage{amsthm}
\usepackage{amsmath}
\usepackage{newlfont}
\usepackage{graphicx}
\usepackage{graphics}
\usepackage{graphpap}
\input xy
\xyoption{all}

\newtheorem*{theorem}{Theorem}
\newtheorem{lemma}{Lemma}

\newtheorem{remark}{Remark}

\newtheorem{corollary}{Corollary}

\begin{document}
\title{Proper trajectories of type $\mathbb{C}^{\ast}$ of a polynomial vector field on $\mathbb{C}^2$}


\author{Alvaro Bustinduy}

\address{Departamento de Ingenier{\'\i}a Industrial \newline
         \indent Escuela Polit{\'e}cnica Superior \newline
         \indent Universidad Antonio de Nebrija \newline
         \indent C/ Pirineos 55, 28040 Madrid. Spain}
         \email{abustind@nebrija.es}

\thanks{Supported by MEC
projects MTM2004-07203-C02-02 and  MTM2006-04785.}

\begin{abstract}  We prove that if a polynomial
vector field on $\mathbb{C}^2$ has a proper and non-algebraic
trajectory analytically isomorphic to $\mathbb{C}^{\ast}$ all its
trajectories are proper, and except at most one which is
contained in an algebraic curve of type $\mathbb{C}$ all of them are of
type $\mathbb{C}^{\ast}$. As corollary we obtain an analytic
version of Lin-Zaidenberg Theorem for polynomial foliations.
\end{abstract}

\maketitle

\section{Introduction}

We shall consider from now on polynomial vector fields on
$\mathbb{C}^2$ with isolated zeroes. Such vector fields $X$
define a foliation by curves $\mathcal{F}_X$ in $\mathbb{C}^2$
with a finite number of singularities (zeros of $X$) that extends
to $\mathbb{CP}^2=\mathbb{C}^2 \cup L_{\infty}$ (see
\cite{Gomez-mont-libro}). Each trajectory $C_z$ of $X$ through a
$z\in\mathbb{C}^2$ with $X(z)\neq{0}$ is contained in a leaf $\mathcal{L}$  of this extended foliation, and
its limit set $lim\,(C_z)$ is defined as $\cap_{\,m\geq 1}\,\,
\overline{\mathcal{L}\setminus K_m}$, where $K_m \subset  K_{m+1}
\subset \mathcal{L}$ is a sequence of compact subsets with
$\cup_{\,m\geq 1}\,\, K_m = \mathcal{L}$. We say that a trajectory
$C_z$ is {\em proper} if its topological closure $\overline{C}_z$
defines an analytic curve in $\mathbb{C}^{2}$ of pure dimension
one, i.e. if the inclusion of $\overline{C}_z$ in $\mathbb{C}^2$
is a proper map. For a proper trajectory $C_z$ its $lim\,(C_z)$ is
either a finite set of points, and $C_z$ is said to be algebraic,
or it contains $L_{\infty}$, and $C_z$ is said to be
non-algebraic. \em In what follows, transcendental will mean
proper and non-algebraic. \em

The important work of Marco Brunella  on the trajectories of a
polynomial vector field with a transcendental planar isolated end
\cite{Brunella-topology} has a remarkable corollary:
If $X$ is a
polynomial vector field on $\mathbb{C}^2$ with a transcendental
trajectory $C_z$ of type $\mathbb{C}$ (``of type" means
analytically isomorphic to) the foliation $\mathcal{F}_X$
in $\mathbb{C}^2$ is equal to the foliation defined by a
constant vector field after an holomorphic automorphism \cite[Corollairie]{Brunella-topology}. In particular any proper
immersion $\gamma$ of $\mathbb{C}$ in $\mathbb{C}^2$ whose image
is contained in a leaf of a polynomial foliation is equal to
$\gamma(t)=(t,0)$ modulo a holomorphic automorphism. That result
can be considered as an Abhyankar-Moh and Suzuki Theorem
(\cite{Abhyankar-moh-crelle} and \cite{Suzuki-japonesa}) for
polynomial foliations \cite[p.\,1230]{Brunella-topology}. In this
note we will study the case of a polynomial vector field  with a
transcendental trajectory of type $\mathbb{C}^{\ast}$. We will
start with \cite[Th\'eor�eme]{Brunella-topology} and apply some
previous results of \cite{Bustinduy-indiana} and
\cite{Bustinduy-ijm} to determine these vector fields. The main
result is the following:
\begin{theorem}
If a polynomial vector field $X$ on $\mathbb{C}^2$ has a
transcendental trajectory of type $\mathbb{C}^{\ast}$, all its
trajectories are proper, and except at most one which is
contained in an algebraic curve of type $\mathbb{C}$ all of them are of
type $\mathbb{C}^{\ast}$.
\end{theorem}

\section{Corollaries}

\begin{corollary}
Any polynomial vector field $X$ on $\mathbb{C}^2$ with  a
transcendental trajectory of type $\mathbb{C}^{\ast}$ has a
meromorphic first integral of type $\mathbb{C}^{\ast}$ which
modulo a holomorphic automorphism is of the form
\begin{equation}\label{seq}
x^m(x^{\ell}y+p(x))^n
\end{equation}
where $m\in\mathbb{Z}^{\ast}$, $n\in\mathbb{N}^\ast$ with
$(m,n)=1$, $\ell\in\mathbb{N}$, $p\in\mathbb{C}[x]$ of degree $<
\ell$ with $p(0)\neq{0}$ if $\ell>0$ or $p(x)\equiv{0}$ if
$\ell=0$.
\end{corollary}
\begin{proof}
According to Masakazu Suzuki  \cite[Th\'eor�eme
II]{Suzuki-springer-largo} for a vector field on $\mathbb{C}^2$
with proper parabolic trajectories  there is always a meromorphic
first integral. In particular for $X$ this integral must be of
type $\mathbb{C}^{\ast}$ and it can be explicitly written applying
Saito-Suzuki Theorem \cite[p.\,527]{Suzuki-anales},
\cite{Saito-osaka}.
\end{proof}

\begin{remark} \em
It follows from Corollary~1 that if $X$ is a polynomial vector
field on $\mathbb{C}^2$ with a transcendental trajectory of type
$\mathbb{C}^{\ast}$  after a holomorphic change of coordinates $\varphi$, the corresponding vector field $\varphi_{\ast} X$ (maybe not polynomial) has a rational first
integral of the form (\ref{seq}). Removing the poles and zeros of codimension one of the differential of (\ref{seq}) one obtains that $\varphi_{\ast} X$ must be of the form
\begin{equation}
\label{poly}
\varphi_{\ast} X =   f \cdot Y = f \cdot \left \{
nx^{l+1}\frac{\partial}{\partial x} - ((m +
nl)x^ly+mp(x)+nx\dot{p}(x))\frac{\partial}{\partial y}
\right\},
\end{equation}
where $f$ is a holomorphic function that never vanishes; and
$m$, $n$, $\ell$ and $p(x)$ are as in (\ref{seq}).
In particular, any foliation $\mathcal{F}_X$ generated by a polynomial vector
field $X$ on $\mathbb{C}^2$ with a transcendental trajectory of type
$\mathbb{C}^{\ast}$  corresponds to the algebriac foliation
generated by the polynomial vector field $Y$ of (\ref{poly}) after a holomorphic automorphism.
\em
\end{remark}

\noindent {\bf Analytic version of Lin-Zaidenberg Theorem for
polynomial vector fields}

\noindent Lin-Zaidenberg Theorem \cite{Lin-zaidenberg} asserts
that any irreducible algebraic curve of type $\mathbb{C}$ in
$\mathbb{C}^2$ is of the form $y^r - a x^{s}=0$, with $(r,s)=1$
and $a\in\mathbb{C}^{\ast}$, after a polynomial change of
coordinates. From our Theorem we obtain the \em analytic version
of this theorem for polynomial foliations: \em

\begin{corollary} Let $\mathcal{C}$ be an
irreducible transcendental curve in $\mathbb{C}^2$ of type
$\mathbb{C}$. If there is a point $p\in \mathcal{C}$ such that
$\mathcal{C}\setminus\{p\}$ defines a trajectory of a polynomial
vector field then $\mathcal{C}=\{y^r - a
x^{s}=0\},\,\,\,r,s\in\mathbb{N}^{+},\,(r,s)=1,\,a\in\mathbb{C}^{\ast},$
up to a holomorphic automorphism.
\end{corollary}
\begin{proof}
As $\mathcal{C}\setminus\{p\}$ is a trajectory of type
$\mathbb{C}^{\ast}$ of a polynomial vector field it must be
contained in a level set of (\ref{seq}) by Corollary 1. If the
level is over $a\neq{0}$, as it is of type $\mathbb{C}$, $\ell=0$
and $m<0$. It is enough define $r=n$ and $s=-m$. If the level set
is over zero, necessarily it is a line: $\{x=0\}$ or also $\{y=0\}$ if
$\ell=0$, which has the required form with $r=s=1$ after a
rotation.
\end{proof}

\begin{remark}\em The classification of H. Saito in \cite{Saito-osaka2}
contains polynomials of this form:
$$P=4((xy+1)^{2}+ y)(x(xy+1)+1)^{2}+1$$
Such a $P$ has two singular fibers: $P^{-1}(0)$ and $P^{-1}(1)$ . One of them, $P^{-1}(1)$, is a disjoint union
of two curves of type $\mathbb{C}^{\ast}$, and another, $P^{-1}(0)$, is an irreducible curve of type $\mathbb{C}^{\ast}$. The generic fiber of $P$
is of type $\mathbb{C}\setminus\{0,1\}$. In particular, our Theorem implies that \em if there is a polynomial vector field with a holomorphic
first integral of the form $P\circ \varphi$ with $\varphi$ a holomorphic automorphism then either $\varphi$ is a polynomial automorphism or $(P\circ\varphi)^{-1}(0)$ and
$(P\circ\varphi)^{-1}(1)$ are contained in algebraic curves.
\em
\end{remark}

\section{Proof of Theorem}

Let $C_z$ be the transcendental trajectory of $X$ of type
$\mathbb{C}^{\ast}$. It defines a leaf $L$ of
$\mathcal{F}_X$ of type $\mathbb{C}^{\ast}$ with a transcendental
planar isolated end $\Sigma$ (\em see \em
\cite[Lemma~4.1]{Bustinduy-indiana}). We can apply
\cite[Th\'eor�eme]{Brunella-topology} and conclude that there
exists a polynomial $P$ with generic fiber of type $\mathbb{C}$ or
$\mathbb{C}^{\ast}$ (that we will call of type $\mathbb{C}$ or
$\mathbb{C}^{\ast}$, respectively) such that $\mathcal{F}_X$ is
$P$--complete.
Let us  recall from \cite{Brunella-topology} that $\mathcal{F}_X$ is is
\em $P-$\,complete \em  if there exists a finite set
$\mathcal{Q}\subset\mathbb{C}$ such that for all $t\not\in
\mathcal{Q}$ : $(i)$ $P^{-1}(t)$ is transverse to
$\mathcal{F}_{X}$, and $(ii)$ there is a neighbourhood
$U_t$ of $t$ in $\mathbb{C}$ such that $P:P^{-1}(U_t)\to U_t$ is a
holomorphic fibration and the restriction of ${\mathcal{F}_{X}}$
to $P^{-1}(U_t)$ defines a local trivialization of this fibration.

As noted in \cite[p.\,1229]{Brunella-topology} (\em see \em also
\cite[Remark~2.2]{Bustinduy-indiana}) the set $\mathcal{Q}$
associated to $P$ consists of the critical values of $P$ together
with the regular values of $P$ in which some of the components of
the corresponding fiber are not transversal to $\mathcal{F}_X$,
and then they are invariant by $\mathcal{F}_X$. Thus every leaf of
$\mathcal{F}_X$ is either disjoint from $P^{-1}(\mathcal{Q})$ or
else is  contained in it.

\subsection{$P$ of type $\mathbb{C}$}

If $\mathcal{F}_X$ is  $P$--\,complete with $P$ of type
$\mathbb{C}$ it can be determined explicitly. According to
Abhyankar-Moh and Suzuki Theorem (\cite{Abhyankar-moh-crelle} and
\cite{Suzuki-japonesa}), up to a polynomial automorphism, we
assume that $P=x$. It is pointed out in
\cite[pp.\,1230]{Brunella-topology} (\em see \em also
\cite[Lemma\,2.6]{Bustinduy-indiana}) that a foliation
$\mathcal{F}_X$ on $\mathbb{C}^2$ which is $x$--\,complete is
generated by a vector field of the form:
\begin{equation}\label{RicattiI}
a(x)\frac{\partial}{\partial x} +[b(x)y+c(x)]
\frac{\partial}{\partial y},\,\,a,b,c\in\mathbb{C}[x].
\end{equation}

As $C_z$ is covered by $\mathbb{C}$ the projection of the universal
covering map  by $P$  defines a map from $\mathbb{C}$ to
$a(x)\neq{0}$, and according Picard Theorem we may assume $a(x)=\lambda x^N$
with $\lambda \in\mathbb{C}^{\ast}$. Remark that $C_z\not\subset
\{x=0\}$ since $C_z$ is not algebraic. In fact as $C_z$ is of type
$\mathbb{C}^{\ast}$ it holds $N>0$.

\begin{lemma}
If $L$ is the leaf of $\mathcal{F}_X$ defined by $C_z$, the leaves of
$\mathcal{F}_{X}$ different from the one contained in $\{x=0\}$ are
defined by the sets $f_{\alpha}(L)$, where $f_{\alpha}$ are the
translations in $\mathbb{C}^2$ of the form: $(x,y)\to (x+\alpha,
y)$, $\alpha \in \mathbb{C}$.
\end{lemma}

\begin{proof}
Let us divide (\ref{RicattiI}) by $\lambda x^{N}$. The system
obtained can be integrated explicitly as a linear equation: For a
fixed $z=(x,y)\in\mathbb{C}^2$, from the first equation
$x(t)=t+x$. By substitution of it in the second equation
if $y=uv$ we get
$$(uv)'=uv'+u'v=\bar{b}(x(t))uv+\bar{c}(x(t)),$$
with $\bar{b}(x)= b(x)/\lambda x^{N}$ and $\bar{c}(x)=c(x)/\lambda x^{N}$. If
$v'=\bar{b}(x(t))v$ then $v(t)=e^{\int \bar{b}(x(s))ds}$ and
$u'v=\bar{c}(x(t)).$ Hence
$$u(t) = \mu + \int \bar{c}(x(u))\,e^{-\left[\int \bar{b}(x(s))ds\right]}\,du,\,\,\,\,
\,\,\mu\in\mathbb{C}.$$ The trajectories of $X$ different from one
contained in $\{x=0\}$ are the subsets in $\mathbb{C}^2$ defined
by the images $\gamma_{(x,y)}(\mathbb{C}\setminus\{-x\})$
of the (mulivaluated) parametrizations
$$
\gamma_{(x,y)}(t)=\left(t+x,\,\, \left\{ y+\int^{t}
\bar{c}( u +x)\, e^{-\left[\int^{u}\bar{b}( s + x)ds
\right]}du \right\} e^{\int^{t}\bar{b}( s + x)ds}\right).
$$

Let $L'$ be a leaf of $\mathcal{F}_{X}$ such that $L'\neq L$ and
$L' \not\subset \{x=0\}$. There is at least one (in fact there are
lots of them) $z_{1}= (x_{1}, y_{1})\in C_z$ such that
$\{y=y_{1}\}\cap L' \neq \emptyset$. If
$z_2=(x_{2},y_{1})\in\{y=y_{1}\}\cap L'$ then $L'=C_{z_2}=
\gamma_{(x_{2},y_{1})}(\mathbb{C}\setminus\{-x_{2}\})$. As
$L=\gamma_{(x_{1},y_{1})}(\mathbb{C}\setminus\{-x_{1}\})$
since $C_{z_{1}}=C_z$ we see that $L'=f_{\alpha}(L)$ with
$\alpha=x_1-x_2$.
\end{proof}

As $L$ is proper by hypothesis and the maps $f_{\alpha}$ are
linear automorphisms \em the leaves of $\mathcal{F}_X$ different
from the one defined by $\{x=0\}$ are proper and biholomorphic to
$L$, i.e. of type $\mathbb{C}^{\ast}$. \em
\subsection{$P$ of type $\mathbb{C}^{\ast}$}
The situation is completely different to the previous one, since
in this case there are many distinct polynomials of type
$\mathbb{C}^{\ast}$ after a polynomial automorphism. According to
Saito and Suzuki  (\cite{Saito-osaka} and \cite{Suzuki-anales}),
up to a polynomial automorphism, we may assume that $P=x^
{m}(x^{\ell}y+p(x))^{n}$, where $m,n\in\mathbb{N}^\ast$ with
$(m,n)=1$, $\ell\in\mathbb{N}$, $p\in\mathbb{C}[x]$ of degree $<
\ell$ with $p(0)\neq{0}$ if $\ell>0$ or $p(x)\equiv{0}$ if
$\ell=0$.

\subsubsection*{{\bf New coordinates}} By the relations
$ x=u^n\,\,\,\,\textnormal{and}\,\,\,\,x^{\ell}y+p(x)=v \,u^{-m},$
it is enough to take the rational map $H$ from $u\neq {0}$ to
$x\neq{0}$ defined by
\begin{equation}\label{relaciones}
(u,v)\mapsto (x,y)=(u^n, {u^{-(m+n\ell)}} [v-u^m p(u^n)])
\end{equation}
in order to get $P\circ H(u,v)=v^n$.

It follows from the proof of \cite[Proposition 3.2]
{Bustinduy-indiana} that $H^{\ast}\mathcal{F}$ is a Riccati
foliation $v$-complete having $u=0$ as invariant line. Still more,
according to \cite[Lemma\,2]{Bustinduy-ijm} at least one of the
irreducible components of $P$ over $0$ must be a
$\mathcal{F}_{X}$-invariant line. \em Therefore we may assume that $\{x=0\}$ is invariant by
$\mathcal{F}_{X}$. \em As $H$ is a finite \em regular \em covering map from
$u\neq 0$ to $x\neq 0$, it implies that each component of
$H^{-1}(C_z)$ is of type $\mathbb{C}^{\ast}$ and then covered by
$\mathbb{C}$. Thus according to Picard's Theorem
\begin{equation} \label{hest}
\begin{split}
H^{\ast} X = &  u^{k}\cdot Z \\
           =& u^{k} \cdot \left\{a(v)u \frac{\partial}{\partial u} +
cv^N\frac{\partial}{\partial v}\right\},
\end{split}
\end{equation}
where $k\in\mathbb{Z}$, $a\in\mathbb{C}[v]$, $c\in\mathbb{C}$, and
$N\in\mathbb{N}^{+}$.

\subsubsection*{{\bf The global one form of times.}} Let us take the one-form $\eta$
obtained when we remove the codimension one zeros and poles of $d
P(x,y)$.  The contraction of $\eta$ by $X$, $\eta(X)$, is a
polynomial, which vanishes only on components of fibres of $P$
since $X$ has only isolated singularities.  Then, up to
multiplication by constants:
\begin{equation} \label{polinomio}
\eta(X)= x^{\alpha}\cdot{(x^{\ell} y + p(x))}^{\beta}
\end{equation}
where $\alpha\in\mathbb{N}^{+}$ (since $\{x=0\}$ is invariant) and $\beta\in \mathbb{N}$.
If we define $\tau=[1/\eta(X)]\cdot\eta$, this one-form on $\eta(X)\neq{0}$ coincides locally along
each trajectory of $X$  with the \em differential of times \em
given by its complex flow. It is called the global one-form of
times for $X$. Moreover $\tau$ can be easily calculated attending
to (\ref{polinomio}) as
\begin{equation}\label{tau}
\tau= \dfrac{x(x^{\ell} y + p(x))}{\eta(X)} \cdot \dfrac{d P}{P}.
\end{equation}
In $(u,v)$ coordinates  we then get
\begin{equation}\label{rho1}
\varrho=H^{\ast}\tau=
\dfrac{u^{m(\beta-1)-n(\alpha-1)}}{v^{\beta-1}}\,\cdot \dfrac{d
v^n}{v^n}
\end{equation}
It holds that $\varrho(H^{\ast} X)\equiv{1}$. Since $\varrho -
1/(u^k \cdot cv^{N})\,dv$ contracted by $H^{\ast}X$ is identically
zero  and we can assume that there is no rational first integral, up to multiplication by constants
\begin{equation}\label{rho2}
\varrho=1/(u^k \cdot cv^{N})\,dv.
\end{equation}
Therefore, (\ref{rho1}) and (\ref{rho2}) must be equal and thus
$k$ of (\ref{hest}) can be explicitly calculated:
$k=n(\alpha-1)-m(N-1)$. Finally, let us observe that  for any path
$\epsilon$ contained in a trajectory of $X$ from $p$ to $q$ that
can be lifted by $H$ as $\tilde{\epsilon}$,
$\int_{\tilde{\epsilon}}\,\varrho$ represents the complex time
required by the flow of $X$ to travel from $p$ to $q$.

\subsubsection*{{\bf Existence of a meromorphic first integral}}
Our aim is to prove that there is an explicit meromorphic first
integral for $X$. We will obtain that as a consequence of the
following lemmas:

\begin{lemma} It holds that $n|k$, $n|(N-1)$ if $N>1$, and $a\in\mathbb{C}[z^{n}]$.
\end{lemma}

\begin{proof}
We assume that $\beta=N$ and $\alpha\in \mathbb{N}^{+}$ in
(\ref{rho1}). Let us observe that $X$ can be explicitly calculated
as
\begin{equation}\label{Y}
X=u^{k}\cdot H_{\ast} (a(v)u\frac{\partial}{\partial u} +cv^N
\frac{\partial}{\partial v})= u^{k} \cdot DH(u,v) \cdot \left(
\begin{array}{c}
a(v)u \\
cv^N
\end{array}
\right)
\end{equation}
where
$$
DH(u,v) =\left(
\begin{array}{cc}
nu^{n-1} & 0 \\
      \\
\dfrac{n\ell u^mp(u^n)-u^{n+m}p'(u^n)-(m + n\ell)v}{u^{m+n\ell+1}}
& \dfrac{1}{u^{m+n\ell}}
\end{array}
\right)
$$
and $u={x}^{1/n}$ and $v={x}^{m/n}\,(x^{\ell}y + p(x))$.

\medskip

Remark  that $a(0)\neq{0}$. Otherwise $X$ had not isolated
singularities since $N>0$. The first component $n
{x}^{(k+n)/n}a({x}^{m/n}\,(x^{\ell}y + p(x)))$ of (\ref{Y}) must
be a polynomial. Then $n|k$. On the other hand $n|(N-1)$ when
$N>1$ since $k=n(\alpha-1)-m(N-1)$ and $(m,n)=1$. It implies that
$a\in\mathbb{C}[z^{n}]$.
\end{proof}
\begin{lemma}\label{lema2}
Let $v_0\neq{0}$. The trajectories of $H^{\ast}X$ except the
horizontal ones and the line $\{u=0\}$ are parameterized by maps
$\sigma(w_0,t)$, where $w_0$ is a fixed point and $\sigma$ is a
multivaluated holomorphic map defined on $\mathbb{C}^{\ast} \times
\mathbb{C}^{\ast}$ of the form
\begin{equation}\label{familia}
\sigma(w,t)=(u(w,t),v(w,t))=
(we^{\int_{v_0}^t\frac{a(z)}{cz^{N}}\,dz}, t).
\end{equation}
\end{lemma}
\begin{proof} Let us take the local solution through
$(u(w_0,v_0),v(w_0,v_0))$, with $w_0\in \mathbb{C}^{\ast}$, of
$1/c(v) \cdot Z$ extending by analytic continuation along paths in
$\mathbb{C}^{\ast}$. This map is defined as $\sigma(w_0,t)$ with
$\sigma$ equals (\ref{familia}) (\em see \em \cite[Section
2]{Bustinduy-jde}).
\end{proof}
\begin{lemma}\label{lema3}
$X$ has a multivaluated meromorphic first integral.
\end{lemma}
\begin{proof}
The one-form of (\ref{familia}), that we denote by $\omega$, has a
fraction expansion
\begin{equation}
\label{desarrollofraccion} \dfrac{a(z)}{cz^{N}}\,dz= \left( s(z) +
\frac{A_{1}}{z}+\frac{A_{2}}{{z}^2}+ \cdots+\frac{A_{N}}{{z}^{N}}\right)
dz ,\,\,\,
\end{equation}
where $s(z)\in\mathbb{C}[z]$, and $A_{i}\in\mathbb{C}^{\ast}$, for $1
\leq i \leq N$. Let us fix
\begin{equation}
\label{Gamma} \Gamma (z) =  e^{\,\,\bar{s}(z)} \cdot
             e^{ \,\,\lambda_{1}
\textnormal{log} \, z + \frac{\lambda_{2}}{z}+
\cdots+\frac{\lambda_{N}}{{\,\,z}^{N-1}}}
\end{equation}
where $\bar{s}(z)=\int^z s(t)dt$, and $\lambda_{1}= A_{1}$ and
$\lambda_{i}=A_{i}/(-i+1)$ for $2 \leq i \leq N$. If we substitute
(\ref{desarrollofraccion}) in (\ref{familia}), after explicit
integration of $\omega$, one has that $\sigma(w,t)$ is of the form
$\left(w \cdot \Gamma(t)/\Gamma(v_0), \,t \right)$. Then
\begin{equation}
\label{primerauv} F(u,v)=\dfrac{u}{\Gamma(v)}
\end{equation}
is a first integral of $H^{\ast}X$. Finally, we can express
(\ref{primerauv}) in terms of $x$ and $y$ by (\ref{relaciones}),
$$
G(x,y)= \frac{x^{1/n}}{\Gamma(x^{m/n}\cdot(x^{\ell}y+p(x)))}\,,
$$
and thus obtain a (multivaluated meromorphic) first integral of
$X$.
\end{proof}

\begin{lemma}
$N=1$, $\lambda_{1}=p/q\in\mathbb{Q}$ and $\bar{s}\in
\mathbb{C}[z^n]$
\end{lemma}
\begin{proof}
When $N>1$ the function $\Gamma(v)$ has an essential singularity at
$v=0$ (for definition of essential singularity of a multivaluated map see \cite[p.\,7]{Huku}).
On the other hand, (\ref{desarrollofraccion}) and (\ref{Gamma}) imply that $\Gamma(v)$ is solution of the differential equation
$$
\dfrac{w'}{w}= \dfrac {v^Ns(v)+v^{N-1}A_{1}+\cdots+A_{N}}{v^{N}}
$$
This differential equation is of the form
\begin{equation}
\label{DE3} v^{N}w'= \dfrac{R(v,w)}{S(v,w)}
\end{equation}
with $R(v,w)= w(v^Ns(v)+v^{N-1}A_{1}+\cdots+A_{N})$ and
$S(v,w)\equiv{1}$ verifying: $a)$ $R(v,w)$ is a polynomial in
$w$ whose coefficients are holomorphic around $v=0$,
$b)$ $R(0,w)$ and $S(0,w)$ are not identically zero, and $c)$
$R(v,w)$ and $S(v,w)$ have not common roots when $v=0$. From
\cite[Th\'eor\`eme 1,\,p.\,99]{Huku} then $\Gamma(v)$
verifies the \em Picard's Property: \em $\Gamma(v)$ takes in any
punctured disk centered at $v=0$ all the values in $\mathbb{C}$
except the \em zero, \em which corresponds with the unique \em
principle characteristic value \em of (\ref{DE3})
\cite[p.\,34]{Huku} given by the solutions of $R(0,w)=0$. Therefore each level
of (\ref{primerauv}), and then each component of $H^{-1}(C_z)$,
accumulates $v=0$. It implies that $C_z$ accumulates
$x^{\ell}y +p(x)=0$  by the equations of $H$ (\ref{relaciones}) what is impossible due to
properness of $C_z$. Hence $N=1$.

Let us
show that $\lambda_{1}\in\mathbb{Q}$.
From (\ref{desarrollofraccion}) as $\omega$  has a pole of order
one at $v=0$ we can assume that it is $\lambda_{1} / z \,dz$ after
a biholomorphism in a neighborhood of $v=0$ fixing it
\cite{Martinet-Ramis}. This way we may suppose that $F(u,v)= u /
v^{\lambda_1}$.

\noindent $\bullet$ If
$\lambda_{1}\in\mathbb{R}\setminus\mathbb{Q}$ each component of
$H^{-1}(C_z)$ is contained in a real subvariety of dimension
three  \cite[p.\,120]{Loray}. Hence $C_z$ is not proper projecting by $H$.

\noindent $\bullet$ If
$\lambda_{1}\in\mathbb{C}\setminus\mathbb{R}$ each component of
$H^{-1}(C_z)$ must accumulate  $\{u=0\}$ and $\{v=0\}$ \cite[p.\,120]{Loray}. In particular $C_z$
accumulates $x^{\ell}y +p(x)=0$ by the equations of $H$ (\ref{relaciones}) what
again gives us a contradiction with properness of $C_z$.

Finally, $zs(z)=a(z)-a(0)$ implies $\bar{s}\in\mathbb{C}[z^n]$ since
$a\in\mathbb{C}[z^{n}]$ by Lemma 2.
\end{proof}
As a consequence of the above lemmas taking $\lambda_{1}=p/q$ we obtain that
$$
G^{nq}=\dfrac{x^{q}}{e^{\,\,nq\,\bar{s}(x^
{m}(x^{\ell}y+p(x))^{n})}{[x^ {m}(x^{\ell}y+p(x))^{n}]}^{p}}
$$
with $x^ {m}(x^{\ell}y+p(x))^{n}$ as in (\ref{seq}) is a
meromorphic first integral of type $\mathbb{C}^{\ast}$ for $X$ \em
up to a polynomial automorphism. \em Therefore all the
trajectories of $X$ are proper, and except at most the one contained in $x=0$ all of them
are of type $\mathbb{C}^{\ast}$.

\begin{remark} \em According to \S 3.2 any polynomial vector field $X$ with a transcendental trajectory of type $\mathbb{C}^{\ast}$ defining a foliation  $P-\,$complete with $P$ of type $\mathbb{C}^{\ast}$ must be proportional to a complete vector field. It is enough to take in (\ref{Y}) $k=0$ to obtain complete vector fields in the cases (i.2) and (i.3) of \cite[Theorem 1.1]{Bustinduy-indiana}.
\end{remark}

\section*{Acknowledgments}
I want to thank the referee for his suggestions that have improved
this paper. In particular, he pointed out to me Remark~2.

\newpage

\bibliographystyle{plain}
\bibliographystyle{amsalpha}

\def\cprime{$'$} \def\cprime{$'$}

\end{document}